\documentclass[]{interact}
\usepackage{amsmath,amsfonts,amssymb,array,mathdots}

\usepackage{epstopdf}
\usepackage[caption=false]{subfig}

\usepackage[numbers,sort&compress]{natbib}
\bibpunct[, ]{[}{]}{,}{n}{,}{,}
\makeatletter
\def\NAT@def@citea{\def\@citea{\NAT@separator}}
\makeatother

\theoremstyle{plain}
\newtheorem{theorem}{Theorem}[section]

\newtheorem{proposition}{Proposition}[section]

\theoremstyle{definition}
\newtheorem{definition}{Definition}[section]
\newtheorem{example}{Example}[section]
\newtheorem{remark}{Remark}[section]

\usepackage{graphicx}
\usepackage{pgfplots}
\usepackage{tikz}

\DeclareMathOperator{\gph}{gph}
\DeclareMathOperator{\Sol}{Sol}

\DeclareMathOperator{\PVVI}{PVVI}
\DeclareMathOperator{\PVP}{PVP}
\DeclareMathOperator{\Po}{P}
\DeclareMathOperator{\VOP}{VOP}

\DeclareMathOperator{\VI}{VI}

\DeclareMathOperator{\Stat}{Stat}
\DeclareMathOperator{\ri}{ri}
\DeclareMathOperator{\R}{\mathbb{R}}

\DeclareMathOperator{\VVI}{VVI}
\DeclareMathOperator{\inte}{int}

\begin{document}
	\title{An Application of the Tarski-Seidenberg Theorem with Quantifiers to Vector Variational Inequalities}

\author{\name{Vu Trung Hieu\textsuperscript{a}\thanks{CONTACT Vu Trung Hieu. Email: hieuvut@gmail.com}}
	\affil{\textsuperscript{a}Division of Mathematics, Phuong Dong University,\\  171 Trung Kinh Street, Cau Giay, Hanoi, Vietnam}}

\maketitle

\begin{abstract}
We study the connectedness structure of the proper Pareto solution sets, the Pareto solution sets, the weak Pareto solution sets of  polynomial  vector variational inequalities, as well as the connectedness structure of the efficient solution sets and the weakly efficient solution sets of polynomial vector optimization problems. By using the Tarski-Seidenberg Theorem with quantifiers, we are able to prove that these solution sets are semi-algebraic without imposing the Mangasarian-Fromovitz constraint qualification on the system of constraints. 
\end{abstract}

\begin{keywords}
	Polynomial vector variational inequality; polynomial vector optimization; solution set; connectedness structure; semi-algebraic set
\end{keywords}

\section{Introduction}

Introduced by Franco Giannessi \cite{G80} in 1980, the concept of vector variational inequality has received a lot of attention from researchers. It is well known that vector variational inequality is an effective tool \cite{LKLY98} to investigate the solution existence, structure of the solution sets, and solution stability of vector optimization problems.

Several questions concerning to the connectedness structure of the solution sets of monotone affine vector variational inequalities raised in a paper of Yen and Yao \cite{YY2011}.
Recently, applying several results from real algebraic geometry \cite{BCF98} (see also \cite{Coste02} and \cite{HaPham2016}) on semi-algebraic sets and a scalarization method (see, e.g., \cite{LKLY98}), Huong, Yao, and Yen \cite{HYY2015b} have established several new results on the connectedness structure of the solution sets of polynomial vector variational inequalities whose constraint sets satisfy the Mangasarian-Fromovitz constraint qualification \cite{Manga_Fro_1967} at every feasible point. One of their theorems asserts that the solution sets have finitely many connected components. It is worthy to stress that the powerful approach used in \cite{HYY2015b} was first proposed and employed by the authors in \cite{HYY2015a}. The interested reader is referred to the survey paper \cite{Yen2016} for an elementary introduction to some new results on vector variational inequalities in connection with vector optimization problems. Observe that the notion of polynomial vector variational inequality was given in \cite{HYY2015b}. 

By using the Tarski-Seidenberg Theorem in the first-order formula language, called the Tarski-Seidenberg Theorem in the third form, we will directly prove that the Pareto solution sets, the weak Pareto solution sets, and the proper Pareto solution sets of a polynomial vector variational inequality are semi-algebraic without imposing the Mangasarian-Fromovitz constraint qualification on the system of constraints. Thus, the present paper develops an idea suggested by Kim, Pham, and Tuyen \cite{KimPhamTuyen2016}, and gives some refinements and extensions for the results of Huong, Yao, and Yen  \cite{HYY2015b}. 

In recent years, the usage of semi-algebraic multifunctions in optimization theory (see, e.g., \cite{DP2011} and \cite{DL2013}) has produced important results. Herein, we will prove that the basic multifunction associated to a polynomial vector variational inequality is semi-algebraic. 

Our paper can be considered as a new attempt to study the connectedness structure of the solution sets of polynomial vector variational inequalities. Some results concerning the connectedness structure of the solution sets and the stationary point sets of polynomial vector optimization problems are also obtained.

We are indebted to Prof. H.V. Ha for a valuable discussion leading to this investigation. Note that the Tarski-Seidenberg Theorem in the third form has been employed in \cite[Remark 3.2]{KimPhamTuyen2016} for improving some results of \cite{HYY2015b} on polynomial vector optimization problems.

The present paper has three sections. Section 2 is devoted to some definitions, notations, and auxiliary results on vector variational inequalities and semi-algebraic sets. The main results on the connectedness structure of the solution sets of polynomial vector variational inequalities and polynomial vector optimization problems are shown in Section 3.

\section{Preliminaries}

The scalar product of $x, y$ from $\R^n$ is denoted by $\langle x,y\rangle$. Let $K\subset\R^n$ be a nonempty subset and $F: K\to\R^n$ a vector-valued function. The \textit{variational inequality} defined by $F$ and $K$ is the problem
$${(\VI)} \quad {\text{Find}}\ \; x\in K \ \;  \text{such that}\ \; \langle F(x),y-x\rangle\geq 0,\  \forall y\in K.$$
The corresponding solution set is denoted by $\Sol(\VI)$.

Given $m$ vector-valued functions $F_l: K\to\R^n$ ($l=1,\dots, m$), we put $F=(F_1,\dots, F_m)$ and define
$$
F(x)(u)=( \langle F_1(x),u\rangle, \cdots, \langle F_m(x),u\rangle), \ \forall x\in K, \ u\in\R^n.
$$
Denoting the nonnegative orthant of  $\R^m$ by $C=\R^m_+$, we call the  problem
$${ (\VVI)} \quad {\text{Find}}\ \, x\in K \ \; \text{such that}\ \, F(x)(y-x)\nleq_{C\setminus\{0\}}0,\ \, \forall y\in K,$$ where the inequality means $-F(x)(y-x)\notin C\setminus\{0\}$, the \textit{vector variational inequality} defined by $F$ and $K$. The solution set of $(\VVI)$, called the \textit{Pareto solution set}, is abbreviated to $\Sol(\VVI)$. One associates to $(\VVI)$ the problem
$$(\VVI)^w \quad {\text{Find}}\ \, x\in K \ \; \text{such that}\ \, F(x)(y-x)\nleq_{\inte C}0, \ \, \forall y\in K,$$
where the inequality means $-F(x)(y-x)\notin \inte C$ and ${\inte}C$ stands for the interior of $C$. The solution set of $(\VVI)^w$, called the \textit{weak Pareto solution set}, is denoted by $\Sol^w(\VVI)$.

The relation $\inte C \subset C\setminus\{0\}$ implies $\Sol(\VVI)\subset\Sol^w(\VVI)$. Moreover, by the closedness of $\R^m\setminus(\inte C)$, if $K$ is closed, then $\Sol^w(\VVI)$ is closed.

Consider the standard simplex $$\Delta=\Big\{(\xi_1,\dots,\xi_m)\in \R^m_+:
\sum_{l=1}^m\xi_l=1\Big\},$$ whose relative interior is given by 
$$\ri\Delta=\{\xi\in\Delta: \xi_l>0, l=1,...,m\}.$$
For each $\xi\in\Delta$, consider the variational inequality 
$$(\VI)_{\xi} \quad  \text{Find }  x\in K \ \text{ such that }  \Big\langle  \sum_{l=1}^m\xi_lF_l(x),y-x\Big\rangle  \geq 0, \ \, \forall y\in
K,$$
and denote its solution set by $\Sol(\VI)_{\xi}$. If  $x\in {\Sol(\VI)}_\xi$ for some $\xi\in\ri\Delta$, then $x$ is said to be a {\it proper Pareto solution} of $(\VVI)$. The proper Pareto solution set of $(\VVI)$ is abbreviated to $\Sol^{pr}(\VVI)$.  By definition, one has 
\begin{equation}\label{pr}
\Sol^{pr}(\VVI)=\bigcup_{\xi\in{\ri}\Delta}{\Sol(\VI)}_\xi.
\end{equation}

The relationships between the solution sets of $(\VVI)$ are as follows.

\begin{theorem}\label{scalarization} {\rm (See \cite{LKLY98})} If $K$ is closed and convex, then
\begin{equation}\label{scalar}
\Sol^{pr}(\VVI)\subset {\Sol(\VVI)}\subset {\Sol^w}{(\VVI)}=\bigcup_{\xi\in\Delta}{\Sol(\VI)}_\xi.
\end{equation}
\end{theorem}

As in \cite{HoaPhuongYen2005}, we associate to $(\VVI)$ the \textit{basic multifunction} $\Phi:\Delta \rightrightarrows \R^n$ with
$\Phi(\xi):= \Sol(\VI)_{\xi}$. If $K$ is closed and convex, according to Theorem \ref{scalarization}, $\Sol^{pr}(\VVI)=\Phi(\ri\Delta)$ and $\Sol^{w}(\VVI)=\Phi(\Delta)$. So, the basic multifunction $\Phi$ is an effective tool to investigate the solution sets of $(\VVI)$.

To proceed furthermore,  we need to recall several results on semi-algebraic sets and functions. 
\begin{definition} {\rm (See \cite{Coste02})} {\rm A set in $\R^n$ is called  \textit{semi-algebraic} if it is the union of finitely
		many subsets of the form
		\begin{equation*}\label{basicsemi}
		\big\{x\in \R^n:f_1(x)=...=f_\ell(x)=0,g_{\ell+1}(x)<0,\dots,g_m(x)<0\big\},
		\end{equation*}
		where $\ell,m$ are natural numbers, and $f_1,\dots, f_\ell,\ g_{\ell+1},\dots,g_m$ belong to the ring $\R[x_1,\dots, x_n]$ of polynomials with real coefficients.} 
\end{definition}

The semi-algebraic property is preserved under the operations of taking finite union, intersection, sets minus, and topological closure within the family of semi-algebraic sets. Moreover, the Tarski-Seidenberg Theorem asserts that these sets are stable under linear projections. 

To present the Tarski-Seidenberg Theorem in the third form, which is the main tool of our subsequent investigations, we have to describe semi-algebraic sets via the language of first-order formulas \cite{Coste02}. 

A \textit{first-order formula} (with parameters in $\R$) is
obtained by the induction rules:
\begin{enumerate}
	\item[(i)] If $p \in\R[X_1, . . . , X_n]$, then $p > 0$ and $p= 0$ are first-order formulas;
	\item[(ii)] If $P,Q$ are first-order formulas, then ``$P$ \textit{and} $Q$'', ``$P$ \textit{or} $Q$'', ``\textit{not} $Q$'', which are denoted respectively by $P \wedge Q$, $P \vee Q$, and $\neg Q$, are first-order formulas;
	\item[(iii)] If $Q$ is a first-order formula, then $\exists X\, Q$ and $\forall X\, Q$, where $X$ is a variable ranging over $\R$, are first-order formulas.
\end{enumerate}

Formulas obtained by using only the rules (i) and (ii) are called \textit{quantifier-free
	formulas}. Clearly,  a subset $S \subset \R^n$ is semi-algebraic if and only if there is a quantifier-free formula $Q_S(X_1,...,X_n)$ such that
$$(x_1,...,x_n) \in S\ \; \text{if and only if }\; Q_S(x_1,..., x_n).$$
In this case, $Q_S(X_1,...,X_n)$ is said to be a \textit{quantifier-free formula defining $S$.}

\begin{theorem}\label{Tar_Sei3} {\rm (The Tarski-Seidenberg Theorem in the third form; see \cite{Coste02})} If $Q(X_1,...,X_n)$ is a first-order formula, then the set $$S=\big\lbrace (x_1,...,x_n)\in\R^n: \ Q(x_1,...,x_n) \big\rbrace $$ is a semi-algebraic set. 
\end{theorem}

By applying Theorem \ref{Tar_Sei3}, we now show that the solution set of a polynomial variational inequality is  semi-algebraic. The reasoning used here will be applied repeatedly later on. 

\begin{example} \rm If $F_l(x) \in\R[x_1, . . . , x_n], l=1,...,m,$ and $K$ is a semi-algebraic set, then $\Sol(\VI)$ is a semi-algebraic set. Indeed, the solution set is defined by 
	$$\Sol(\VI)=\big\{x\in\R^n:x\in K,\ \forall y\in K,  \langle F(x),y-x\rangle\geq 0\big\}.$$
	Let $Q_{\Sol}(X)$ be the first-order formula
	$$Q_K(X) \wedge \left[\forall Y \left( \left[ Q_K(Y)\wedge \langle F(X),Y-X\rangle\geq 0\right]\vee \left[ \neg Q_K(Y)\right]\right)  \right],$$
	where $Q_K(X)$ is a quantifier-free formula defining $K$. Clearly, $x\in \Sol(\VI)$ if and only if $Q_{\Sol}(x)$. So, $\Sol(\VI)$ is a semi-algebraic set by Theorem \ref{Tar_Sei3}.
\end{example}

A topological space $S$ is said to be \textit{connected} if it cannot be represented as $S=U\cup V$, where $U$ and $V$ are nonempty disjoint open sets of $S$. A nonempty subset $A\subset S$ is said to be a \textit{connected component} of  $S$ if $A$, equipped with the induced topology, is connected and it is not a proper subset of any connected subset of $S$. 
The topological space $S$ is said to be {\it path connected} if, for every $x, y$ in $S$, there is a continuous mapping $\gamma:[0,1]\to S$ such that $\gamma(0)=x$ and $\gamma(1)=y$. 
Any path connected topological space is connected, but the converse is not true. However, if $S\subset\R^n $ is a semi-algebraic set, then the two connectedness properties are equivalent \cite[Theorem 2.4.5  and Proposition 2.5.13]{BCF98}. It is well known that any semi-algebraic set has finitely many connected components \cite[Theorem 2.4.5]{BCF98}. 

\section{Main results}
\subsection{Polynomial vector variational inequalities}
\begin{definition} {\rm (See \cite{HYY2015b})} {\rm One says that $(\VVI)$ is a \textit{polynomial}  vector variational inequality and denotes it by $(\PVVI)$ if all the components of the operators 
		$F_{l}=(F_{l1}, . . . , F_{ln}), l=1,...,m,$ are polynomials and the constraint set $K$ is semi-algebraic.}
\end{definition}

Without assuming that $K$ is closed or $K$ satisfies a regularity condition, let us prove that the Pareto solution set and the weak Pareto solution set of $(\PVVI)$ are semi-algebraic. The following result is a significant improvement of the first and the third assertions of Theorem 3.3 in \cite{HYY2015b}. The proof is based on some new arguments.

\begin{theorem}\label{finite} Both sets $\Sol^w(\PVVI)$ and $\Sol(\PVVI)$  are semi-algebraic. So, each of them has finitely many connected components and every component is path connected.
\end{theorem}
\begin{proof} We first prove that $\Sol^w(\PVVI)$ is semi-algebraic. Recall that
	$$\Sol^w(\PVVI)=\big\{x\in K:-F(x)(y-x)\in \R^m\setminus (\inte C),\forall y\in K\big\}.$$
	As $\inte C=\{(\xi_1,\dots,\xi_m)\in \R^m:
	\xi_l> 0,l=1,...,m\}$,  one has $a\in\R^m\setminus (\inte C)$ if and only if $a$ belongs to the union of the closed half-spaces $$\left\{(\xi_1,...,\xi_m)\in\R^m:\xi_l\leq 0\right\}\quad (l=1,...,m)$$ of $\R^m$. Therefore,
	\begin{equation}\label{weak_sol_set} \Sol^w(\PVVI)=\Big\lbrace x\in K\;:\; \bigvee_{l=1}^{m} \big[F_l(x)(y-x)\geq 0\big], \;\forall y\in K\Big\rbrace.	
	\end{equation}
	For each $l\in\{1,...,m\}$, since all the components of $F_l(x)$ are polynomials in the variables $x_1,...,x_n$, the scalar product $\langle F_l(x),y-x\rangle$ is a polynomial in the variables $x_1,...,x_n,y_1,...,y_n$. Hence, the expressions $$\langle F_l(X),Y-X\rangle\geq 0\quad (l=1,...,m),$$ are quantifier-free formulas. In addition, since $K$ is a semi-algebraic set in $\R^n$, there exists a quantifier-free formula in $n$ variables $Q_K(X_1,...,X_n)$ defining $K.$ So, thanks to \eqref{weak_sol_set}, $x\in \Sol^w(\PVVI)$ if and only if $Q^w(x_1,...,x_n)$, where $Q^w(X_1,...,X_n)$ is the following first-order formula
	$$Q_K(X) \wedge\Big[\forall Y\Big(\Big[Q_K(Y) \wedge\Big(\bigvee_{l=1}^{m} \big[F_l(X)(Y-X)\geq 0\big]\Big)\Big] \vee \big[ \neg Q_K(Y)\big] \Big)\Big] .$$
	According to Theorem \ref{Tar_Sei3}, $\Sol^w(\PVVI)$ is a semi-algebraic set.
	
	Similarly, the Pareto solution set can be represented as follows:
	$$\Sol(\PVVI)=\big\{x\in K\;:\; -F(x)(y-x)\in (\R^m\setminus C)\cup \{0\},\; \forall y\in K\big\}.$$
	Since $C=\{(\xi_1,\dots,\xi_m)\in \R^m:
	\xi_l\geq 0, l=1,...,m\}$,  one has $a\in\R^m\setminus C$ if and only if $a$ belongs to the union of the open half-spaces $$\big\{(\xi_1,...,\xi_m)\in\R^m:\xi_l< 0\big\}\quad (l=1,...,m).$$ Hence, the set $\Sol(\PVVI)$ can be described as
	$$\Big\lbrace x\in K: \Big[ \bigvee_{l=1}^{m} \big(F_l(x)(y-x)> 0\big)\Big] \vee \Big[ \bigwedge_{l=1}^{m} \big(F_l(x)(y-x)= 0\big)\Big], \forall y\in K\Big\rbrace.$$
	So, we have $x\in \Sol(\PVVI)$ if and only if $Q(x_1,...,x_n)$, where $Q(X_1,...,X_n)$ is the following first-order formula in $n$ variables
	$$Q_K(X) \wedge\Big[ \forall Y\Big(\Big[Q_K(Y) \wedge \big[A(X,Y) \vee B(X,Y)\big]\Big]\vee \big[\neg Q_K(Y)\Big]\Big)\Big] $$
	with $$A(X,Y):=\bigvee_{l=1}^{m} \big[ F_l(X)(Y-X)> 0\big], \ B(X,Y):= \bigwedge_{l=1}^{m} \big[ F_l(X)(Y-X)= 0 \big].$$
	Therefore, $\Sol(\PVVI)$ is a semi-algebraic set by Theorem \ref{Tar_Sei3}.
\end{proof}
\begin{remark}\rm
	The convexity of $K$ has not been used in the preceding proof.
\end{remark}

The next statement refines and extends the second assertion of Theorem~3.3 in \cite{HYY2015b}. 

\begin{theorem}\label{finite_pr} The set $\Sol^{pr}(\PVVI)$ is semi-algebraic. So, it has finitely many connected components and each component is path connected.
\end{theorem}
\begin{proof} Since $\ri\Delta=\big\{\theta=(\theta_1,\dots,\theta_m)\in\Delta: \theta_l>0,l=1,...,m\big\}$ is a semi-algebraic set in $\R^m$, there exists a quantifier-free formula $Q_{\ri\Delta}(\Theta)$ in the variables $\Theta_1,...,\Theta_m$ defining $\ri\Delta.$
The set $\Sol^{pr}(\PVVI)$ defined by
 $$\Big\lbrace  x\in K:\exists\theta\in\ri\Delta, \Big\langle  \sum_{l=1}^m\theta_lF_l(x),y-x\Big\rangle  \geq 0, \forall y\in
	K \Big\rbrace .$$
	So, $x\in \Sol^{pr}(\PVVI)$ if and only if $Q^{pr}(x_1,...,x_n)$, where $Q^{pr}(X_1,...,X_n)$ is the following first-order formula
	$$Q_K(X) \wedge\Big[ \exists \Theta\ Q_{\ri\Delta}(\Theta)\Big]  \wedge \Big[ \forall Y \Big(\big[ Q_K(Y) \wedge C(\Theta,X,Y)\big]\vee \big[ \neg Q_K(Y)\big] \Big)\Big]$$
	with 
	\begin{equation}\label{C_expression}
	C(\Theta,X,Y):=\Big\langle  \sum_{l=1}^m\Theta_lF_l(X),Y-X\Big\rangle \geq 0.
	\end{equation}
	Hence, Theorem \ref{Tar_Sei3} allows us to conclude that the set $\Sol^{pr}(\PVVI)$ is semi-algebraic.
\end{proof}

Let $S\subset \R^m$ be a semi-algebraic set. Following \cite{DL2013}, we say that a multifunction $\Psi: S\rightrightarrows \R^n$ is \textit{semi-algebraic} if its graph
$$\gph \Psi=\big\{(x,y)\in \R^m\times \R^n: x\in S, y\in \Psi(x)\big\}$$
is a semi-algebraic set in  $\R^{m+n}$.
\begin{proposition}
The basic multifunction $\Phi:\Delta \rightrightarrows \R^n$ of the problem $(\PVVI)$ is semi-algebraic. 
\end{proposition}
\begin{proof} Since the standard simplex  $\Delta$ is a semi-algebraic set in $\R^m$, there is a quantifier-free formula $Q_{\Delta}(\Theta)$ in variables $\Theta_1,...,\Theta_m$ defining $\Delta.$	We see that
	$$\gph\Phi=\Big\lbrace  (\theta,x)\in \Delta\times K:\Big\langle  \sum_{l=1}^m\theta_lF_l(x),y-x\Big\rangle  \geq 0,\forall y\in
	K \Big\rbrace.$$
	Consider the first-order formula $Q_{\Phi}(\Theta,X)$ defined by	
	$$Q_K(X) \wedge Q_{\Delta}(\Theta) \wedge \Big[ \forall Y \Big(  \big[ Q_K(Y) \wedge C(\Theta,X,Y) \big]\vee \big[ \neg Q_K(Y)\big] \Big)\Big] ,$$
	where $C(\Theta,X,Y)$ has been defined in \eqref{C_expression}.	Clearly, $(\theta,x)\in \gph\Phi$ if and only if $Q_{\Phi}(\theta,x)$. According to Theorem \ref{Tar_Sei3}, $\gph\Phi$ is a semi-algebraic set. This means that the multifunction $\Phi$ is semi-algebraic.
\end{proof}

\begin{example}
	Consider the bicriteria polynomial variational inequality $(\Po)$ in $\R^2$, where 
	$$F_1(x)=\begin{bmatrix}0 & 1 \\ 
	-1 &  0
	\end{bmatrix}\begin{bmatrix}x_1 \\ 
	x_2
	\end{bmatrix}+\begin{bmatrix}
	0\\ 
	-1
	\end{bmatrix},F_2(x)=\begin{bmatrix}0& -1 \\ 
	1 & 0
	\end{bmatrix}\begin{bmatrix}x_1 \\ 
	x_2
	\end{bmatrix}+\begin{bmatrix}
	0\\ 
	-1
	\end{bmatrix},$$ 
	and 
	$$K=\{(x_1,x_2)\in\R^2:g(x_1,x_2)=x_1^2-x_2-4\leq 0\}.$$
	Clearly, $K$ satisfies all the assumptions of \cite[Proposition 1.3.4]{FaPa03} and the Abadie constraint qualification{\rm(ACQ)} at each point of $K$. So, we have
	$x\in\Sol(\VI)_{\xi}$ if and only if there exist $\lambda\in \R_+$ such that 
	$$\xi_1F_1(x)+(1-\xi_1)F_2(x)+\lambda\nabla g(x)=0,\ \lambda g(x)=0,\ g(x) \leq 0.$$
	The first equality means that
	\begin{equation}\label{equation_matrix}
	\begin{bmatrix}0 & 2\xi_1- 1 \\ 
	1-2\xi_1 &  0
	\end{bmatrix}\begin{bmatrix}x_1 \\ 
	x_2
	\end{bmatrix}+\begin{bmatrix}
	0\\ 
	-1
	\end{bmatrix}+\begin{bmatrix}
	2\lambda x_1\\ 
	-\lambda
	\end{bmatrix}=\begin{bmatrix}
	0\\ 
	0
	\end{bmatrix}.
	\end{equation}
	Consider the following two cases.
	\begin{enumerate}
		\item[(a)] $x\in \inte K$. Since $g(x)<0$ and $\lambda g(x)=0$, we have $\lambda=0$. Hence, \eqref{equation_matrix} becomes
		$$\begin{bmatrix}0& 2\xi_1- 1 \\ 
		1-2\xi_1& 0
		\end{bmatrix}\begin{bmatrix}x_1 \\ 
		x_2
		\end{bmatrix}=\begin{bmatrix}
		0\\ 
		1
		\end{bmatrix}.$$
		Combining these facts, we obtain
		$$\inte K\cap\Sol^w(\Po)=\Big\lbrace \Big( \frac{1}{1-2\xi_1},0\Big): \xi_1\in \big[0,\frac{1}{4}\big)\cup \big( \frac{3}{4},1\big]  \Big\rbrace.$$
		\item[(b)] $x\in \partial K:= K\setminus\inte K$. We have $g(x)=0$ and $\lambda \geq0$. Solving \eqref{equation_matrix} we obtain
		$$x=\left(\frac{1+\lambda}{1-2\xi_1},\frac{2\lambda^2+2\lambda}{(1-2\xi_1)^2}\right).$$
		Combining these facts, $\partial K\cap\Sol^w(\Po)$ defined by
		$$\left\lbrace\left(\frac{1+\lambda}{1-2\xi_1},\frac{2\lambda^2+2\lambda}{(1-2\xi_1)^2}\right): \xi_1\in \big[\frac{1}{4},\frac{1}{2}\big)\cup \big(\frac{1}{2},\frac{3}{4}\big]  \right\rbrace,$$
		where $\lambda=\sqrt{1-4(1-2\xi_1)^2}$.
	\end{enumerate}
	\begin{center}
		\begin{tikzpicture}
		\begin{axis}[
		xlabel = $x_1$,
		ylabel = $x_2$,]
		\addplot [thin, dashed, domain=-3:3, 
		samples=100, 
		color=gray,]{x^2 - 4};
		\addplot [very thick, domain=-3:-2, 
		samples=100, 
		color=gray,]{x^2 - 4};
		\addplot [very thick, domain=2:3, 
		samples=100, 
		color=gray,]{x^2 - 4};
		\addplot [very thick, domain=1:2, 
		samples=100, 
		color=gray,]{0};
		\addplot [very thick, domain=-2:-1, 
		samples=100, 
		color=gray,]{0};
		\addlegendentry{$x_1^2 -x_2=4$}
		\addlegendentry{$\Sol^w(\Po)$}
		\end{axis}
		\end{tikzpicture}
		
		{\textbf{Figure 1.}\ \, The weak Pareto solution set $\Sol^w(\Po)$.}
	\end{center}
	So, the basic multifunction is given by
	$$\Phi(\xi_1,1-\xi_1)=\left\{\begin{array}{cl}
	\left( \frac{1}{1-2\xi_1},0\right) & \quad \hbox{ if } \ \ \xi_1\in \Big[0,\frac{1}{4}\Big)\cup \Big(\frac{3}{4},1\Big],  \\
	\emptyset & \quad \hbox{ if } \ \ \xi_1=\frac{1}{2}, \\
	\left( \frac{1+\lambda}{1-2\xi_1},\frac{2\lambda^2+2\lambda}{(1-2\xi_1)^2}\right) &  \quad \hbox{ if } \ \ \xi_1\in \Big[\frac{1}{4},\frac{1}{2}\Big)\cup \Big(\frac{1}{2},\frac{3}{4}\Big].  \\
	\end{array}\right.$$
	According to Theorems \ref{finite} and \ref{finite_pr}, both solution sets $\Sol^w(\Po)$
	and $\Sol^{{pr}}(\Po)$ are semi-algebraic. Moreover, each
	solution set has two connected components and each component is path connected and unbounded.
\end{example}

\subsection{Polynomial vector optimization problems}
In this section, we will study polynomial 
vector optimization problems and obtain some topological properties of the solution sets and the stationary point sets.
First, let us specify several solution concepts from \cite{Jahn2011} to polynomial vector
optimization problems.

Let there be given a nonempty closed semi-algebraic subset  $K\subset\R^n$ and polynomial functions $f_1,\dots,f_m:\R^n\to\R$. The vector minimization problem with the constraint set $K$ and the vector objective function $f:=(f_1,\dots,f_m)$ is written formally as follows:
\begin{equation*}\label{VOP} (\PVP)\quad {\rm Minimize}\ \; f(x)\ \ {\rm subject \ to}\ \  x\in K.
\end{equation*}
A point $x\in K$ is said to be  a \textit{Pareto solution} of $(\PVP)$ if  $f(y)-f(x)$ does not belong to $-\R^m_+\setminus\{0\}$
for all $y\in K$. It
is said to be  a {\it weak Pareto solution} of $(\PVP)$ if  $f(y)-f(x)$ does not belong to $-\inte\R^m_+$ for all $y\in K$. The Pareto solution set and the weak Pareto solution set of
$(\PVP)$ are respectively abbreviated to $\Sol(\PVP)$ and $\Sol^w(\PVP)$.

We see that $a\notin -\R^m_+\setminus\{0\}$ if and only if $a$ belongs to the union of the open half-spaces $\big\{(\xi_1,...,\xi_m)\in\R^m:\xi_l>0\big\}$ $(l=1,...,m)$ of $\R^m$. Similarly, $a\notin -{\rm int}\R^m_+$ if and only if $a$ belongs to the union of the closed half-spaces $\big\{(\xi_1,...,\xi_m)\in\R^m:\xi_l\geq 0\big\}$ $ (l=1,...,m)$ of $\R^m$.

\begin{theorem}
	Both sets $\Sol^w(\PVP)$ and $\Sol(\PVP)$  are semi-algebraic. So, each of them has finitely many connected components and every component is path connected.
\end{theorem}
\begin{proof} The weak Pareto solution set $\Sol^w(\PVP)$ can be represented as 
	$$\Sol^w(\PVP)=\Big\lbrace x\in K: \bigvee_{l=1}^{m}[f_l(y)-f_l(x)\geq 0],\forall y\in K\Big\rbrace.$$
	For every $l=1,\dots,m$, since $f_l(x)$ is a polynomial in the variables $x_1,...,x_n$, $f_l(y)-f_l(x)$ is a polynomial in the variables $x_1,...,x_n,y_1,...,y_n$. So the expression $f_l(Y)-f_l(X)\geq 0$ is a quantifier-free formula.  Since $K$ is a semi-algebraic set in $\R^n$, there exists a quantifier-free formula in $n$ variables $Q_K(X_1,...,X_n)$ defining $K.$ So, $x$ belongs to $\Sol^w(\PVP)$ if and only if $Q^w(x_1,...,x_n)$, where $Q^w(X_1,...,X_n)$ is the following first-order formula in $n$ variables
	$$Q_K(X) \wedge\Big[  \forall Y \Big( \Big[  Q_K(Y) \wedge\Big(\bigvee_{l=1}^{m} [f_l(Y)-f_l(X)\geq 0]\Big)\Big]\vee \big[ \neg Q_K(Y)\big]  \Big) \Big]  .$$
	Therefore, by Theorem \ref{Tar_Sei3}, $\Sol^w(\PVVI)$ is a semi-algebraic set.
	
	Analogously, the Pareto solution set  $\Sol(\PVP)$ is described as
	$$\Big\lbrace x\in K: \Big[ \bigvee_{l=1}^{m} \big(f_l(y)-f_l(x)>0\big)\Big] \vee \Big[ \bigwedge_{l=1}^{m} \big(f_l(y)-f_l(x)= 0\big)\Big],\ \forall y\in K\Big\rbrace.$$
	Hence, $x\in \Sol(\PVP)$ if and only if $Q(x_1,...,x_n)$, where $Q(X_1,...,X_n)$ is the following first-order formula in $n$ variables
	$$Q_K(X)  \wedge \Big(  \forall Y \Big[\Big( Q_K(Y) \wedge \left[ A(X,Y) \vee B(X,Y)\right]\Big)\vee \Big( \neg Q_K(Y)\Big) \Big] \Big) .$$
	with $$A(X,Y):=\bigvee_{l=1}^{m} \big[ f_l(Y)-f_l(X)>0\big],$$ and $$B(X,Y):= \bigwedge_{l=1}^{m} \big[ f_l(Y)-f_l(X)=0\big].$$
	According to Theorem \ref{Tar_Sei3}, $\Sol(\PVP)$ is a semi-algebraic set.
\end{proof}

\begin{remark}
	It was shown in \cite[Remark 3.2]{KimPhamTuyen2016} that the set of all Pareto values (resp., weak Pareto values) and the set of all Pareto solutions (resp., weak Pareto solutions)  of  an unconstrained polynomial vector optimization problem are semi-algebraic. 
\end{remark}

Put $F_l(x)=\nabla f_l(x)$, with $\nabla f_l(x)$ denoting the gradient of $f_l$ at $x$. For any $\xi\in\Delta$, we consider the parametric variational inequality ${\rm (VI)}_\xi$, which now becomes
$$(\VI)_\xi \quad {\text{Find}}\ \, x\in K \  \text{such that}\  \Big\langle\sum_{l=1}^m\xi_l\nabla f_l( x),y- x\Big\rangle\geq 0,\  \forall y\in K.$$
Assume that $K$ is closed convex. According to \cite[Theorem 3.1(i)]{LKLY98}, if $x$ is a weak Pareto solution of $(\PVP)$, then there is $\xi\in\Delta$ with $x\in {\rm Sol(VI)}_\xi$. 

If $x\in K$ and there exists $\xi\in\Delta$ such that $x\in {\Sol(\VI)}_\xi$, then $x$ is said to be a {\it stationary point} of $(\PVP)$. If $x\in K$ and there is $\xi\in{\rm ri}\Delta$ with $x\in {\Sol(\VI)}_\xi$, then we call $x$ a {\it proper stationary point} of $(\PVP)$. The stationary point set and the proper stationary point set of $(\PVP)$ are respectively abbreviated to $\Stat(\PVP)$ and $\Pr(\PVP)$. From these definitions it follows that
\begin{equation}\label{Pr_stat}
\bigcup_{\xi\in {\ri}\Delta}{\Sol(\VI)}_\xi={\Pr(\PVP)}\subset {\Stat(\PVP)}=\bigcup_{\xi\in\Delta}\Sol(\VI)_\xi.
\end{equation}

\begin{remark}
	Assume that $K$ is convex and all the functions $f_i$ are convex. If there is $\xi\in\Delta$ such that $x\in {\Sol(\VI)}_\xi$, then $x$ is a weak Pareto solution of $(\PVP)$; see \cite[Theorem 3.1(ii)]{LKLY98}. If there is $\xi\in\ri\Delta$ such that $x\in {\Sol(\VI)}_\xi$, then $x$ is a Pareto solution of $(\PVP)$; see \cite[Theorem 3.1(iii)]{LKLY98}. So, we have
	\begin{equation}\label{relation1}
	\bigcup_{\xi\in {\ri}\Delta}{\Sol(\VI)}_\xi\subset{\Sol(\PVP)}\subset \Sol^w(\PVP)=\bigcup_{\xi\in\Delta}{\Sol(\VI)}_\xi.
	\end{equation}
\end{remark}

\begin{proposition}\label{Stat} 
	Both sets $\Stat(\PVP)$ and $\Pr(\PVP)$  are semi-algebraic. So, each of them has finitely many connected components and every component is path connected.
\end{proposition}

\begin{proof} Consider the problem $(\PVVI)$, where $F_l(x)=\nabla f_l(x)$ for $l=1,\dots,m$. 	
	From the equality in \eqref{scalar} and the last equality in \eqref{Pr_stat}, one has  $\Stat(\PVP)=\Sol^w(\PVVI)$. Since Theorem \ref{finite} asserts that $\Sol^w(\PVVI)$ is  semi-algebraic, so is the set $\Stat(\PVP)$.
By \eqref{pr} and the first equality in \eqref{Pr_stat}, the sets $\Pr(\PVP)$ and $\Sol^{pr}(\PVVI)$ coincide. As $\Sol^{pr}(\PVVI)$ is semi-algebraic by Theorem \ref{finite_pr}, so is the set $\Pr(\PVP)$.
\end{proof}
\begin{example}
	Consider the bicriteria optimization problem $(\VOP)$ in $\R^2$, with the constraint $K=\{x\in\R^2:-x_1\leq 0\}$ and two polynomial functions given by $$f_1(x)=\frac{1}{4}x_1^4-x_2,\; f_2(x)=\frac{1}{3}x_2^3-x_1.$$
 It is easy to check that $f_1$ is convex over $\R^2$ and $f_2$ is convex over the closed-half space $K$. So, both of them are convex on $K$. To find the Pareto solution set and the weak Pareto solution set of $(\VOP)$, we will solve the bicriteria variational inequality $(\VVI)$ derived from the optimization problem. 
	The gradients of $f_1,f_2$ at $x$ are given by
	$$\nabla f_1(x)=F_1(x)=\begin{bmatrix}
	x_1^3\\ 
	-1
	\end{bmatrix},\ \nabla f_2(x)= F_2(x)=\begin{bmatrix}
	-1\\ 
	x_2^2
	\end{bmatrix}.$$ 
	Since {\rm(ACQ)} is satisfied at each point of $K$, by \cite[Proposition 1.3.4]{FaPa03}  we know that 
	$x\in\Sol(\VI)_{\xi}$ if and only if there exist $\lambda\in \R_+$ such that the following equation and the inequality are satisfied:
	$$\xi_1F_1(x)+(1-\xi_1)F_2(x)+\lambda\nabla g(x)=0,\ \lambda g(x)=0,\ g(x) \leq 0.$$
	The first equation can be rewritten as
	\begin{equation}\label{equation_matrix_2}
	\begin{bmatrix}
	\xi_1 x_1^3-(1-\xi_1)\\ 
	(1-\xi_1)x_2^2-\xi_1
	\end{bmatrix}+\begin{bmatrix}
	-\lambda\\ 
	0
	\end{bmatrix}=\begin{bmatrix}
	0\\ 
	0
	\end{bmatrix}.
	\end{equation}
	Consider the following two cases:
	\begin{enumerate}
		\item[(i)] $x\in \inte K$. Then we have $g(x)<0$ and $\lambda=0$. Therefore, 
		$$\inte K\cap\Sol^w(\VVI)=\Big\lbrace \Big(\sqrt[3]{\frac{1-\xi_1}{\xi_1}},\sqrt{\frac{\xi_1}{1-\xi_1}}\Big) :\xi_1\in \big(0,1\big)  \Big\rbrace.$$
		\item[(ii)] $x\in \partial K=\big\{x\in\R^2: -x_1=0\big\}$. Then we have $g(x)=0$ and $\lambda\geq 0$. System \eqref{equation_matrix_2} leads to
		$\partial K\cap\Sol^w(\VVI)=\emptyset.$
	\end{enumerate}
Hence, the basic multifunction of $(\VVI)$ is given by
	$$\Phi(\xi_1,1-\xi_1)=\left\{\begin{array}{cl}
	\left(\sqrt[3]{\frac{1-\xi_1}{\xi_1}},\sqrt{\frac{\xi_1}{1-\xi_1}}\right)  & \ \ \hbox{ if } \ \xi_1\in \Big(0,1\Big),  \\
	\emptyset & \ \ \hbox{ if } \ \xi_1\in\{0,1\} . \\
	\end{array}\right.$$
\begin{center}
		\begin{tikzpicture}
		\begin{axis}[axis lines = center,
		xlabel = $x_1$,
		ylabel = $x_2$,]
		\addplot [very thick, domain=0:5, 
		samples=100, 
		color=gray,]{1/sqrt(x^3)};
		\addlegendentry{$\Sol(\VOP)=\Sol^w(\VOP)$}
		\end{axis}
		\end{tikzpicture}
		
		{\textbf{Figure 2.}\ \, The Pareto solution set $\Sol(\VOP)$.}
	\end{center}
	Thus, 	$\Sol^{pr}(\VVI)=\Sol^{w}(\VVI)$. According to \eqref{relation1}, the Pareto solution set and the weak Pareto solution set of $(\VOP)$ coincide. Here we have $$\Sol(\VOP)=\Sol^w(\VOP)=\Big\lbrace (x_1,x_2)\in \R^2: x_2^2=\frac{1}{x_1^3}, x_1>0\Big\rbrace .$$ 
This is an unbounded, connected semi-algebraic set. 
\end{example}

\section*{Acknowledgements}
	The author is indebted to Professor Nguyen Dong Yen for many stimulating conversations and Professor Ha Huy Vui for very helpful comments concerning the Tarski-Seidenberg Theorem in the third form.  The author wishes to thank the referee for a remark leading to Proposition~\ref{Stat}.




\begin{thebibliography}{99}

\bibitem{G80}  Giannessi F: Theorems of alternative,
quadratic programs and complementarity problems. In: Cottle, R.W., Giannessi, F., Lions, J.-L. (eds.): Variational
	inequality and complementarity problems;  New York: Wiley; 1980. p. 151--186.

\bibitem{LKLY98} Lee GM, Kim DS, Lee BS, Yen ND: Vector variational inequalities as a tool for studying vector	optimization problems. Nonlinear Anal. 1998; 34:745--765.

\bibitem{YY2011} Yen ND, Yao J-C: Monotone affine vector variational inequalities.  Optimization 2011; 60:53--68.

\bibitem{BCF98}  Bochnak R, Coste M, Roy MF: Real algebraic geometry. Berlin: Springer; 1998.

\bibitem{Coste02} Coste M: An Introduction to semialgebraic geometry. Institut de Recherche Math\'ematique de Rennes. Universit\'e de Rennes; 2002.

\bibitem{HaPham2016} Ha HV, Pham TS: Genericity in polynomial optimization. Singapore: World Scientific; 2017.

\bibitem{HYY2015b} Huong NTT, Yao J-C, Yen ND: Polynomial vector variational inequalities under polynomial constraints and applications. SIAM J. Optim. 2016; 26:1060--1071.

\bibitem{Manga_Fro_1967}  Mangasarian OL, Fromovitz  S: The Fritz John necessary optimality conditions in the presence of equality and inequality constraints. J. Math. Anal. Appl. 1967; 17:37--47. 

\bibitem{HYY2015a} Huong NTT, Yao J-C, Yen ND: Connectedness structure of the solution sets of vector variational inequalities. Optimization 2017; 66:889--901.

\bibitem{Yen2016} Yen ND: An introduction to vector variational inequalities and some new results. Acta Math. Vietnam. 2016; 41:505--529.

\bibitem{KimPhamTuyen2016} Kim DS, Pham TS, Tuyen NV: On the existence of Pareto solutions for semi-algebraic vector optimization problems. Math. Program. 2018; Available from: https://doi.org/10.1007/s10107-018-1271-7.

\bibitem{DP2011} Daniilidis A, Pang JCH: Continuity and differentiability of set-valued maps revisited in the light of tame geometry. J. London Math. Soc. 2011; 83:637--658.

\bibitem{DL2013}  Drusvyatskiy D,  Lewis AS:  Semi-algebraic functions have small subdifferentials.  Math. Program. 2013; 140:5--29.

\bibitem{HoaPhuongYen2005} Hoa TN, Phuong, TD, Yen ND: On the parametric affine variational inequality approach to linear fractional vector optimization problems. Vietnam J. Math. 2005; 33:477--489.

\bibitem{FaPa03} Facchinei F, Pang J-S: Finite-dimensional variational inequalities and complementarity problems. {Vols. I and II.} New York: Springer;  2003.


\bibitem{Jahn2011} Jahn J: Vector optimization: theory, applications, and extensions. 2nd ed. Berlin: Springer; 2011.

\end{thebibliography}


\end{document}